\documentclass[12pt]{article}

\usepackage{amsmath}
\usepackage{fixltx2e}
\usepackage{etex}
\usepackage{xspace}
\usepackage{lmodern}
\usepackage[T1]{fontenc}
\usepackage{textcomp}
\usepackage[utf8]{inputenc}
\usepackage{microtype}
\usepackage{hyperref}

\newcommand*{\cs}[1]{\texttt{\textbackslash#1}}
\makeatletter
\newcommand*{\cmd}[1]{\cs{\expandafter\@gobble\string#1}}
\makeatother

\usepackage{amssymb}
\usepackage{amsthm}
\usepackage{amsmath}

\newtheorem{definition}{Definition}
\newtheorem{theorem}{Theorem}
\newtheorem{remark}{Remark}

\newtheorem{lemma}{Lemma}
\newtheorem{corollary}{Corollary}

\newtheorem{property}{Property}

\begin{document}


\title{On the  transformations linearizing  isochronous centers of Hamiltonian  systems }

\author{Guangfeng Dong \\{\small Department of Mathematics, Jinan University,}\\{\small Guangzhou 510632, China,
donggf@jnu.edu.cn(Corresponding author)}\\ \\Yuyi Zhang \\{\small Department of Mathematics, Jinan University,}\\{\small Guangzhou 510632, China,
951537882@qq.com }}

\date{\today}

\maketitle
\begin{abstract}
In this paper we study the transformations linearizing  isochronous centers
of planar Hamiltonian differential systems with polynomial Hamiltonian functions $H(x,y)$
having only isolated singularities.
Assuming the origin is an isochronous center lying on the level curve $L_0$ defined by $H(x,y)=0$,
we prove that,
there exists a canonical linearizing transformation
analytic  on a simply-connected open set $\Omega$  with closure $\overline{\Omega}=\mathbb{R}^2$,
if and only if,
 $L_0$ consists of only isolated points;
furthermore, if the origin is the unique center,
then the condition that $L_0$ consists of only isolated points
implies that the corresponding canonical linearizing transformation
can be analytically defined  on the whole plane.
\\
\\
\noindent \textbf{Keywords:} isochronous center; Hamiltonian systems; canonical linearizing transformation;  commuting system.

\noindent \textbf{MSC(2010):} 34C20;  37C10.
\end{abstract}

\section{Introduction and  main results}
\label{intro}

Isochronous center is one kind of the most interesting  singularities of integrable differential systems
and  has been studied extensively for decades
(see, e.g., \cite{C-S,F-R-S-T,L-V} and references therein),
especially for polynomial Hamiltonian differential systems as follows:
 \begin{align}\label{H-R}
\left(
  \begin{array}{c}
    \frac{{\rm d} x}{{\rm d} t} \\
    \frac{{\rm d} y}{{\rm d} t} \\
  \end{array}
\right)
=\left(
      \begin{array}{r}
        -H_y\\
        H_x \\
      \end{array}
    \right), \ H_x=\frac{\partial H(x,y)}{\partial x},\ H_y=\frac{\partial H(x,y)}{\partial y},
\end{align}
where the Hamiltonian function
$H(x,y)=(x^2 +y^2 )/2+h.o.t.$ is a polynomial of degree $n$ in $\mathbb{R}[x,y]$ such that all of its singularities are isolated.
Clearly the origin $O$ is a center and
 its period function  is defined through the period of each periodic orbit inside the period annulus.
If the period function is constant, then the center is  conventionally
called isochronous.

The isochronicity of Hamiltonian systems, as a special case within
 the more general category, is a very subtle problem and has been characterized completely only for very few families.
It is proved in \cite{C-J} that in the potential case the unique polynomial
isochronous center is the linear one.
When  the Hamiltonian function takes the form $H(x, y) =F(x) +G(y)$, it is
proved in  \cite{C-G-M} that
the unique isochronous center  turns out to be the linear one, too.
For polynomial Hamiltonian systems,
due to \cite{Loud},  the quadratic isochronous centers have been classified completely.
In \cite{C-M-V,L-R}, isochronous centers  for $n=4$ are
investigated with a   corresponding classification.
In addition, it has been shown, see, for example,   \cite{C-D,G-G-M-M},
  that Hamiltonian systems have no isochronous centers if they have homogeneous nonlinearities.
Some other related results can be found in, e.g., \cite{F-R-S-T,L-R} and the references therein.

Notably the isochronicity  is closely related with
the linearizibility of a center.
For system (\ref{H-R}), it is well known that(see \cite{M-V})
the origin is  an isochronous center of period $2\pi$ if and only if
there there exists a canonical  transformation
\begin{eqnarray}\label{phi}
\Phi:\ (x,y)\mapsto (u(x,y),v(x,y))
\end{eqnarray}
 analytic in a neighbourhood of the origin
transforming system (\ref{H-R}) to a linear system
\begin{eqnarray}\label{H-R-L}
\left(
  \begin{array}{c}
    \frac{{\rm d} u}{{\rm d} t} \\
    \frac{{\rm d} v}{{\rm d} t} \\
  \end{array}
\right)
=\left(
      \begin{array}{r}
        -v\\
        u\\
      \end{array}
    \right),
\end{eqnarray}
where  $u(x, y)=x+h.o.t.$, $v(x, y)=y+h.o.t.$.
Here a transformation $\Phi$ is called \emph{canonical}
if its Jacobian determinant is constant.
Given the above forms of $u$ and $v$, this constant is equal to $1$, i.e.
\begin{eqnarray*}
\det\left(\frac{\partial (u,v)}{\partial (x,y)} \right)=\det \left(
  \begin{array}{cc}
    u_x & u_y\\
    v_x & v_y\\
  \end{array}
\right)=1.
 \end{eqnarray*}
Note that a canonical linearizing transformation $\Phi$ implies the following equation holds
\begin{eqnarray}\label{uv-H}
\frac{u^2(x,y) +v^2(x,y) }{2}=H(x,y),
\end{eqnarray}
which means that $\Phi$ maps the level curve
\begin{eqnarray}\label{L_h}
L_h=\{(x,y):H(x,y)=h\}
\end{eqnarray}
onto the circle
$$S_h=\{(u,v):u^2 +v^2=2h\}.$$
There may also exist non-canonical transformations changing a Hamiltonian system to another Hamiltonian system, see, e.g., \cite{Saba1}.

In general the analytic transformation $\Phi$ is well defined locally.
The question whether or not it can be defined globally is very important
in the study of the global topological structure of an isochronous center.
For example,
the global existence of a canonical linearizing transformation implies that
there are no isochronous centers for system (\ref{H-R}) with odd $n$(see, e.g. Theorem 1 of \cite{L-R}).
This fact provides strong support to the negative answer to the following open question(see \cite{J-V}):
\emph{Does there exist a planar polynomial Hamiltonian system (\ref{H-R}) with odd $n$ having an isochronous center?}
From the proof of the following theorem,
 we can also obtain the same conclusion to Theorem $1$ of \cite{L-R}(see Corollary \ref{Odd} below).

In this paper,
we study the global existence of a canonical linearizing transformation for system (\ref{H-R}),
and obtain the following two theorems to characterise some properties of the domains where the transformation can be defined well.
\begin{theorem}\label{TM-c-g}
For  system (\ref{H-R}),
if the origin $O$  is  an isochronous center,
then there exists a canonical linearizing transformation $\Phi$
analytic on a simply-connected open set $\Omega \ni O$ with the closure  $\overline{\Omega}  =\mathbb{R}^2$,
if and only if,
 $L_0 $ consists of only isolated points.
\end{theorem}

When the origin is the unique
center of system (\ref{H-R}),
we have a further result as follows.
\begin{theorem}\label{TM-g}
For system (\ref{H-R}),
if  $L_0 $ consists of only isolated points,
and the origin is the unique center which is isochronous,
then there exists a canonical linearizing transformation analytic
on the whole plane.
\end{theorem}

To prove the above theorems,
the key tool is highly based on the existence of a special system commuting with system (\ref{H-R}).
So we first introduce some properties about it.

\section{ Commuting systems}

\begin{definition}
We call that two differential systems on $\mathbb{R}^2$
\begin{eqnarray*}
\left(
  \begin{array}{c}
    \dot{x} \\
    \dot{y} \\
  \end{array}
\right)
=\left(
      \begin{array}{l}
        F_{i}(x,y)\\
        G_{i}(x,y)\\
      \end{array}
    \right),\ i=1,2,
\end{eqnarray*}
commute with each other,
if the corresponding vector fields
$$X_1=(F_{1}(x,y),G_{1}(x,y)),\ \ X_2=(F_{2}(x,y),G_{2}(x,y))$$ commute with each other, that is, their Lie bracket
$$[X_1, X_2]=
\left(
               \begin{array}{cc}
                 \frac{\partial F_2}{\partial x} & \frac{\partial F_2}{\partial y} \\
                 \frac{\partial G_2}{\partial x} & \frac{\partial G_2}{\partial y} \\
               \end{array}
             \right)
\left(
               \begin{array}{c}
                 F_1 \\
                 G_1 \\
               \end{array}
             \right)
-\left(
               \begin{array}{cc}
                 \frac{\partial F_1}{\partial x} & \frac{\partial F_1}{\partial y} \\
                 \frac{\partial G_1}{\partial x} & \frac{\partial G_1}{\partial y} \\
               \end{array}
             \right)
\left(
               \begin{array}{c}
                 F_2 \\
                 G_2 \\
               \end{array}
             \right)
=0.$$
\end{definition}
Taking advantage of the Jacobian matrix
\begin{align*}
J= \frac{\partial (u,v)}{\partial (x,y)} =\left(
  \begin{array}{cc}
    u_x & u_y\\
    v_x & v_y\\
  \end{array}
\right)
\end{align*}
of canonical transformation $\Phi$, we can construct a new system
\begin{align}\label{C-H-R}
\left(
  \begin{array}{c}
   \frac{{\rm d} x}{{\rm d} s} \\
    \frac{{\rm d} y}{{\rm d} s} \\
  \end{array}
\right)
=(J^{\top} J)^{-1}\left(
      \begin{array}{c}
        H_x \\
        H_y  \\
      \end{array}
    \right),
\end{align}
possessing the following  good properties
which can be achieved easily by the directly computation and
the following equations from relation (\ref{uv-H})(or, see \cite{C-S}):
\begin{align*}
J^{\top}
\left(
  \begin{array}{c}
   u \\
   v
  \end{array}
\right)
=
\left(
  \begin{array}{cc}
   u_x & v_x \\
   u_y & v_y
  \end{array}
\right)
\left(
  \begin{array}{c}
   u \\
   v
  \end{array}
\right)
=\left(
      \begin{array}{c}
        H_x \\
        H_y
      \end{array}
    \right),
\end{align*}
and here we use $s$ to represent its time variable
in order to distinguish from the time variable $t$ of system (\ref{H-R}).

\begin{property}\label{property}
Assuming that the canonical linearizing transformation $\Phi$
can be well defined on an open set $E$ containing the origin,
then on $E$ we have

\begin{enumerate}
\item     system (\ref{H-R})  and system (\ref{C-H-R}) can be linearized simultaneously by  $\Phi$, that is,  system (\ref{C-H-R}) can also be transformed by $\Phi$ to a linear system
    \begin{align*}
\left(
  \begin{array}{c}
    \frac{{\rm d} u}{{\rm d} s} \\
    \frac{{\rm d} v}{{\rm d} s} \\
  \end{array}
\right)
=\left(
      \begin{array}{r}
        u\\
        v \\
      \end{array}
    \right);
\end{align*}
\item   system (\ref{C-H-R}) commutes with system (\ref{H-R});
\item   system (\ref{C-H-R}) and system (\ref{H-R}) have the same set of singularities;
\item   the two vector fields induced by system (\ref{C-H-R}) and system (\ref{H-R}) respectively are transversal at any non-singular point.
 \end{enumerate}

\end{property}

\section{Proof of theorems}

\begin{proof}[Proof of Theorem \ref{TM-c-g}]
	
We first prove the necessity.
Let $\Omega$ be a simply-connected open set with $\overline{\Omega}=\mathbb{R}^2$
where the canonical linearizing transformation $\Phi$ can be analytically defined well.
Denote by $N$ the set of non-isolated points on $L_0$.
Suppose $N\neq \emptyset$.
Noticing that $L_0$ is algebraic,
the real plane is divided into at least two disjoint open sets by $N$,
while $\Omega$ is simply-connected,
so $N\nsubseteq \partial\Omega$, that is $N\bigcap \Omega \neq \emptyset$.
Then all points of $N\bigcap \Omega $ are mapped  to the origin of the $(u,v)$-plane by $\Phi$,
which is ridiculous because $\Phi$ is a local homeomorphism on $\Omega$.

For the sufficiency,
denoting by $\Lambda_0$ the set of finite singularities  of  system (\ref{H-R}) except the origin,
we take a polyline $\mathcal{P}$ starting from a  point in $\Lambda_0$ ending at infinity,
 such that the vertex set is $\Lambda_0$ and there has no loop in $\mathcal{P}$.
 If $\Lambda_0=\emptyset$, then $\mathcal{P}=\emptyset$.
We will show that,  $\Phi$ can be defined well on $\mathbb{R}^2-\mathcal{P}$.

We start from a small simply-connected open set  $\Omega_0$
bounded by a period orbit near the origin,
where the canonical transformation $\Phi$ is well defined.
This $\Omega_0$ exists due to that $\Phi$ can be well defined  locally.
Assume $\partial\Omega_0 \subseteq L_{h_0}$ for some a $h_0\neq0$.

For any a non-singular point $p\in L_{h_0}$ of system (\ref{H-R}),
there exists a local invertible change of variables $\tau_p:\ (h,t)\mapsto (x,y)$
defined on a sufficiently small disk $D_p$ centered at $p$,
where $t$ represents  the time variable of system (\ref{H-R}),
such that $\tau_p(h_0,0)=p $ and
$\tau_p(h,0)=\gamma_p $,
where $\gamma_p\subseteq D_p $  is a line segment passing through $p$ transversal to vector field (\ref{H-R})
at $p$, e.g, we can choose $\gamma_p $ tangent  to the  gradient field $(H_x,H_y)$ at $p$.
Clearly $\tau_p$ is analytic locally,
because it is the solution of system (\ref{H-R}) containing a parameter $h$ with the initial value condition
$ (x,y)=(x(h,0),y(h,0))$ at $t=0$,
 which is analytic in $h$ for the reason that it is the restriction of $(x,y)$ on $\gamma_p$.

To avoid too many notations,
we still denote  by $u(h,t)$ and $v(h,t)$ respectively $u(x(h,t),y(h,t))$ and $v(x(h,t),y(h,t))$.

Given a point $p\in \partial\Omega_0 $,
let $A_p$ be the limit set of $\Phi(\gamma_p\bigcap \Omega_0)$,
then we have the following

\begin{lemma}\label{l-inter}
$A_p \bigcap S_{h_0}$ consists of only one point.
\end{lemma}

The proof of this lemma will be given after the proof of
Theorem  \ref{TM-c-g}. Notice that if one sets
$\Phi(p)= A_p \bigcap S_{h_0}$,
then one can see that Lemma \ref{l-inter} provides a way to  extend the definition
of $\Phi$ to the point $p$ continuously along
$\gamma_p\bigcap \Omega_0$.
In fact,   one can further
extend the definition of $\Phi$ to a sufficiently
small neighborhood of  $p$ analytically in the following way.

Let $(U,V)=(f(h_0,t),g(h_0,t))$  be the solution of ordinary differential equations(ODEs for short)
\begin{align}\label{UV1}
\left(
  \begin{array}{l}
    \frac{{\rm d}  U}{{\rm d} t} \\
    \frac{{\rm d}  V}{{\rm d} t}
  \end{array}
\right)
=\left(
  \begin{array}{r}
    -V \\
     U
  \end{array}
\right)
\end{align}
on functions $U(h,t)$ and $V(h,t)$ containing a parameter $h$ with the initial value condition
$$ (f(h_0,0),g(h_0,0))= (u_{p}, v_{p})$$
when $h=h_0$, here $ (u_{p}, v_{p})$ is the coordinate of point $A_p \bigcap S_{h_0}$.
Obviously $f(h_0,t)$ and $g(h_0,t)$ both analytically depend on $t$.
Then we can extend the definition of $\Phi$ to $\partial\Omega_0$ near $p$
by
\begin{equation}\label{Phi-c-e}
\Phi(\varphi^{t}(p))\triangleq (f(h_0,t),g(h_0,t)),
\end{equation}
for sufficiently small $|t|$,
where  $\varphi ^{t}(\cdot)$
represents the flow map associated to system (\ref{H-R}),
i.e. for any point $a \in \Omega_0$,
$\varphi^{t}(a)$ takes the value at $t$  of
the solution of  system (\ref{H-R}) with initial value $a$ when $t=0$.
Clearly this extension is continuous.

Next denote by $(\widetilde{u}(h,t), \widetilde{v}(h,t))$
 the solution of ODEs
\begin{align}\label{UV2}
\left(
  \begin{array}{l}
    \frac{{\rm d}  U}{{\rm d} h}
\\
    \frac{{\rm d}  V}{{\rm d} h}
  \end{array}
\right)
=
\left(
  \begin{array}{r}
   \frac{U}{2h} \\
\frac{V}{2h}
  \end{array}
\right)
\end{align}
containing a parameter $t$ with the initial condition
$$(\widetilde{u}(h_0,t),\widetilde{v}(h_0,t))= (f(h_0,t),g(h_0,t)),$$
where ODEs (\ref{UV2}) comes from the following equations
\begin{align*}
\left(
  \begin{array}{l}
    \frac{{\rm d}  U}{{\rm d} s}
\\
    \frac{{\rm d}  V}{{\rm d} s}
  \end{array}
\right)
=
\left(
  \begin{array}{l}
    U \\
 V
  \end{array}
\right)
\end{align*}
by the change of variable $s\mapsto h$ determined by the relation
$$\frac{U^2+V^2}{2}=\frac{C_1 e^{2s}+C_2 e^{2s}}{2}=h$$ with two constants
 $C_1$ and $C_2$ near $h_0$.
According to the analytic dependence on initial values and parameters of the solution for ODEs,
$\widetilde{u}(h,t)$ and $ \widetilde{v}(h,t)$ both analytically depend on $h$ and $t$
in a sufficiently small neighbourhood $D'_p\subseteq D_p$ of $p$.
Then  we can define a new analytic transformation $\widetilde{\Phi}(x,y)$ by
$$\widetilde{\Phi}(x,y)\triangleq(\widetilde{u}(h(x,y),t(x,y)), \widetilde{v}(h(x,y),t(x,y)),
\ \forall (x,y)\in D'_p.$$

According to Property (P1) and equation (\ref{Phi-c-e}),
$(u(h,t),v(h,t))$ is also a solution of ODEs (\ref{UV2}) on $\Omega$
and coincides with $(\widetilde{u}(h,t), \widetilde{v}(h,t))$ on $\partial\Omega_0$.
By the uniqueness of the solution for the initial value problem of ODEs,
$\widetilde{\Phi}$ is equal to $\Phi$ on $D'_p\bigcap\Omega_0$,
i.e., $\Phi$ can be extended to $\Omega_0 \bigcup D'_p$ analytically.

Now we assume that $\Phi$ has been well defined analytically  on $\Omega'_0=\Omega_0 \bigcup D'_p$.
Let $\Omega_1=\bigcup_{q\in \Omega'_0 } \xi_q$,
where
$\xi_q$ is the trajectory of system (\ref{H-R}) passing through $q\in \Omega'_0 $
such that it is  taken the whole continuous component of the trajectory if $\xi_q\bigcap \mathcal{P}=\emptyset$,
 else a  continuous part from $q$ to the intersection of $\xi_q$ and $\mathcal{P}$(excluding this point).
 Then $\Omega_1$ is still open and simply-connected.
 One can define the transformation $\Phi $ on $\Omega_1$ as follows:

 \begin{eqnarray*}
 \Phi(\varphi^{t}(q))\triangleq\varphi^{t}_{\ast}(\Phi(q)),
\ \forall q\in \Omega'_0,
\end{eqnarray*}
for any $t$ such that $\varphi^{t}(q)\in \Omega_1$,
where $\varphi^{t}_{\ast}(\cdot)$
represents the flow map associated to system (\ref{H-R-L}),
i.e. for any point $b\in \Phi(\Omega'_0)$,
$\varphi^{t}_{\ast}(b)$ takes the value at $t$  of
the solution of  system (\ref{H-R-L}) with initial value $b$ when $t=0$.
It is not difficult to see that this $\Phi $ is  canonical on $\Omega_1$,
because its components $u(h,t)$ and $v(h,t)$ satisfy the ODEs (\ref{UV1}) and (\ref{UV2}).

Note that $\partial\Omega_1$ consists of only the trajectories of system (\ref{H-R}) and points in $\mathcal{P}$.
If $\partial\Omega_1 \not\subseteq \mathcal{P}$,
then one can repeat the above steps at a point  $p\in \partial\Omega_1$ but $p\not\in \mathcal{P}$.
Finally $\Phi$ can be extended to a simply-connected open set $\Omega$ with $\partial\Omega = \mathcal{P}$.
Obviously $\overline\Omega=\mathbb{R}^2$, so the theorem is proved.
\end{proof}

\begin{proof}[Proof of Lemma \ref{l-inter}]
Suppose otherwise, i.e.  suppose there are at least two points $q_1$ and $q_2$ in $A_p\bigcap S_{h_0}$,
then there exists  a continuous arc $\sigma$ of $S_{h_0}$ from $q_1$ to $q_2$  in $A_p\bigcap S_{h_0}$,
for the reasons that $h$ monotonically depends on $s$
and $\Phi$ preserves the orientation of two given vector fields.

For any point $q'\in \sigma,\ q'\not = q_1,q_2$, with coordinate $(u_{q'},v_{q'})$,
denote by $(f'(h_0,t),g'(h_0,t))$  the solution of ODEs (\ref{UV1}) with initial value
$$(f'(h_0,0),g'(h_0,0))=(u_{q'},v_{q'}),$$
then by the solution $(u'(h,t),v'(h,t))$ of ODEs (\ref{UV2}) with initial value
$$(u'(h_0,t),v'(h_0,t))=(f'(h_0,t),g'(h_0,t)),$$
one can obtain another analytic transformation $\Phi'$
defined in a sufficiently small neighbourhood $D_2 \subset D_p$ of $p$ by
$$\Phi':\ (x,y)\mapsto (u'(h(x,y),t(x,y)),v'(h(x,y),t(x,y))).$$

Clearly $\Phi'$ linearizes system (\ref{C-H-R}),
i.e., changes system (\ref{C-H-R}) to a linear system
\begin{align}\label{u'v'-L-N}
\left(
  \begin{array}{l}
    \frac{{\rm d}  u'}{{\rm d} s}
\\
    \frac{{\rm d}  v'}{{\rm d} s}
  \end{array}
\right)
=
\left(
  \begin{array}{r}
   u' \\
    v'
  \end{array}
\right).
\end{align}
Besides, we will show that it also linearizes system (\ref{H-R}).
Assume system (\ref{H-R})  is transformed to the following
\begin{align}\label{u'v'}
\left(
  \begin{array}{l}
    \frac{{\rm d}  u'}{{\rm d} t}\\
    \frac{{\rm d}  v'}{{\rm d} t}
  \end{array}
\right)
=
\left(
  \begin{array}{l}
    F(u',v') \\
   G(u',v')
  \end{array}
\right),
\end{align}
then we assert that $(F,G)=(-v',u')$ for the following reasons.
Note that system (\ref{u'v'}) and the linear system
\begin{align}\label{u'v'-L}
\left(
  \begin{array}{l}
    \frac{{\rm d}  u'}{{\rm d} t}
\\
    \frac{{\rm d}  v'}{{\rm d} t}
  \end{array}
\right)
=
\left(
  \begin{array}{r}
   -v' \\
    u'
  \end{array}
\right),
\end{align}
 both commute with system (\ref{u'v'-L-N}),
and  have a common trajectory $u'^2+v'^2=h_0$,
which results in that all of their  trajectories are the same,
i.e. they are equivalent orbitally.
So there exists a continuous function $\lambda(u',v')$ such that
 $(F,G)=(-\lambda v',\lambda u').$
By the commutativity of system (\ref{u'v'}) and  system (\ref{u'v'-L-N}),
$\lambda$ satisfies $\partial\lambda /\partial u'=\partial\lambda /\partial v'=0$ and the initial condition $\lambda=1$ when $u'^{2}+v'^{2}=h_0$,
thus $\lambda\equiv 1,$
i.e., the assertion has been proved.

Due to that $\Phi'(\gamma_p)$ is continuous at $p$ and
$q'$ is a limit point but not  $q_1,q_2$,
we have $\Phi(\gamma_p)\bigcap\Phi'(\gamma_p)\neq\emptyset$.
For a point $q_3\in \Phi(\gamma_p)\bigcap\Phi'(\gamma_p)$ with coordinate $(u_3,v_3)$,
let $(f_{3}(h_3,t),g_{3}(h_3,t))$ be the solution of ODEs (\ref{H-R-L})
with initial value $(f_{3}(h_3,0),g_{3}(h_3,0))=(u_3,v_3)$,
where $(h_3,0)$ is the coordinate of $\Phi^{-1}(q_3)\bigcap D_2=\Phi'^{-1}(q_3)\bigcap D_2$,
consequently we have obtained two different solutions $(u,v)$ and $(u',v')$ for ODEs (\ref{UV2})
with the same initial value $(f_{3}(h_3,t),g_{3}(h_3,t))$ when $h=h_3$,
this contradicts to the uniqueness of the solution of ODEs with the initial value condition.
\end{proof}

\begin{remark}
Note that the proof of the necessity of Theorem \ref{TM-c-g} is also valid for the more general $H(x,y)$
which can have non-isolated singularities.
Obviously  if $H(x,y)$ is a polynomial of odd degree,
then there must exist non-isolated points on $L_0 $.
So we have the following corollary which is equivalent to Theorem 1 of \cite{L-R}.
\begin{corollary}\label{Odd}
For polynomial Hamiltonian system (\ref{H-R})  with odd $n$,
if the origin is isochronous,
then the corresponding canonical linearizing transformation can not be defined well on the whole plane.
\end{corollary}
\end{remark}

\begin{proof}[Proof of Theorem \ref{TM-g}]
By Theorem \ref{TM-c-g},
we can assume  the canonical linearizing transformation $\Phi$ is well defined
in a simply-connected open set $\Omega$ with $\partial\Omega=\mathcal{P}$.
The aim of the theorem is to show $\Lambda_0=\emptyset$.
Suppose otherwise.
Let $p_0\in \Lambda_0$ is the beginning point of $\mathcal{P}$,
then $p_0$ is neither a center by the assumption of the theorem,
 nor a node due to that system (\ref{H-R}) is Hamiltonian,
therefore in $\Omega$ there exists at least one separatrix $\xi_0$ passing through $p_0$.
Clearly by the proof of Theorem \ref{TM-c-g},
the boundary $\mathcal{P}$ can be replaced by a piecewise smooth curve with the same vertex set $\Lambda_0$.
So, without loss of generality,
we can assume $\xi_0\subseteq \mathcal{P}$  in a sufficiently small neighbourhood $D_0$ of $p_0$.
Then there exists another trajectory $\xi_1$ of system (\ref{H-R}) tending to $p_0$
when $t\rightarrow +\infty$ or $t\rightarrow -\infty$.
This implies that there exists a trajectory $\eta$ of system (\ref{C-H-R})  tending to $p$
 in $D_0 - \xi_0$ by Property \ref{property}.

By the commutativity of system (\ref{H-R}) and system (\ref{C-H-R}),
for any time $t$, $\varphi^{t}(\eta)$ is still a trajectory of system (\ref{C-H-R}).
Due to that $p_0$ is singular of system (\ref{H-R}),
for any point $b\in \eta$ sufficiently closed to $p_0$,
the length of the trajectory of system (\ref{H-R}) from $b$ to $\varphi^{t}(b)$
vanishes when $b\rightarrow p_0$,
i.e., $\varphi^{t}(\eta)$ also tends to $p_0$.
This implies $p_0$ is a singularity of node type for system (\ref{C-H-R}),
consequently it is a center of system (\ref{H-R}),
which leads a contradiction.
\end{proof}

\section*{Acknowledgements}
This work is supported by NSFC 11701217 of China,  NSF 2017A030310181 of Guangdong(China),
and  Research Fund 21616312 of Jinan University(China).

\end{document}